\documentclass[review]{elsarticle}
\usepackage{lineno,hyperref}
\usepackage{mathrsfs}
\usepackage{amsfonts}
\usepackage{amssymb}
\usepackage{amsmath}
\usepackage{hyperref}
\usepackage{amsthm}
\usepackage{cases}
\usepackage{arydshln}
\usepackage[usenames,dvipsnames]{color}
\usepackage{graphicx}%
\newtheorem{thm}{Theorem}[section]

\newtheorem{lem}[thm]{Lemma}

\newtheorem{conj}[thm]{Conjecture}

\newcommand{\tr}{\text{Tr}}

\modulolinenumbers[5]
\journal{Linear Algebra Appl.}

\begin{document}
	
	\begin{frontmatter}
		
		\title{A new proof of Lee's conjecture on the Frobenius norm via the matrix Cauchy-Schwarz inequality}
		
		 \author{Teng Zhang\fnref{1footnote}}
		 
		\address{School of Mathematics and Statistics, Xi'an Jiaotong University, Xi'an 710049, China}
		\fntext[1footnote]{Email: teng.zhang@stu.xjtu.edu.cn}
		\begin{abstract}
			     In 2010, Eun-Young Lee conjectured that if $A,B$ are two $n\times n$ complex matrices and $\left|A\right|, \left|B\right|$ are the absolute values of $A, B$, respectively, then
			     \[
			     \|A+B\|_F\le \sqrt{\dfrac{1+\sqrt{2}}{2}}\|\left|A\right|+\left|B\right|\|_F,
			     \]
			     where $\|\cdot\|_F$ is the Frobenius norm of matrices. This conjecture has been  proven by Lin and Zhang [J. Math. Anal. Appl. 516 (2022) 126542] by studying inequalities for the angle between two matrices induced by the Frobenius inner product. In this paper, we present a new proof of the same result, relying solely on the Cauchy-Schwarz inequality.
		\end{abstract} 
	\end{frontmatter}
	\section{Introduction}
Let $M_n$ denote the algebra of all $n\times n$ complex matrices. For $A\in M_n$, the trace of 
$A$, denoted as $\tr\, A$, is defined as the sum of its main diagonal entries. Furthermore, the trace of 
$A$ coincides with the sum of its eigenvalues (counted with algebraic multiplicity) over an algebraically closed field. The operator absolute value of $\left|A\right|$ is defined as $\left|A\right|:=\left(A^*A \right)^\frac{1}{2} $, where $A^*$ represents the conjugate transpose of $A$.   The singular values of $A$, ordered as $s_1(A)\ge \ldots\ge s_n(A)$, characterize the spectral structure of $A$. The Frobenius norm (also known as the Schatten 2-norm) $\|A\|_F$ is given by:
\[
\|A\|_F=\left( \tr \,   \left|A\right|^2\right)^\frac{1}{2}=\left( \sum_{j=1}^ns_j^2(A)\right)^\frac{1}{2}. 
\]
More generally, for any $p\ge 1$, the Schatten $p$-norm $\|\cdot\|_p$ of $A$ is defined as
\[
\|A\|_p=\left( \tr \,  \left|A\right|^p\right)^\frac{1}{p} =\left( \sum_{j=1}^ns_j^p(A)\right)^\frac{1}{p}.
\]

Eun-Young Lee \cite{Lee10}  inquired about the optimal value $c_p$ satisfying
\[
\|A+B\|_p\le c_p\|\left|A\right|+\left|B\right|\|_p
\]
holds for any $A,B\in M_n$. She \cite[p. 584]{Lee10} pointed out even $c_2$ seems difficult to compute and conjectured $c_2=\sqrt{\tfrac{1+\sqrt{2}}{2}}$.
\begin{conj}\label{thm}
	Let $A,B\in M_n$. Then
	\begin{eqnarray}\label{e1}
		   \|A+B\|_F\le \sqrt{\dfrac{1+\sqrt{2}}{2}}\|\left|A\right|+\left|B\right|\|_F.
	\end{eqnarray}
\end{conj} By proving an inequality for the angle between two matrices induced by the Frobenius inner product, Lin and Zhang \cite{LZ22} established the validity of Conjecture \ref{thm}, subsequently demonstrating the optimality of inequality (\ref{e1}) through the canonical matrix pair: \[
A=\begin{bmatrix}
	1&0\\
	0&0
\end{bmatrix}, B=\begin{bmatrix}
1&0\\
0&0
\end{bmatrix}\begin{bmatrix}
\cos\alpha&-\sin\alpha\\
\sin\alpha&\cos\alpha
\end{bmatrix} \text{ with } \cos\alpha=\sqrt{2}-1.
\]

This article provides an elementary proof of Conjecture \ref{thm}, utilizing solely the matrix form of the Cauchy-Schwarz inequality.
\section{Main results}
The following two trace inequalities play an important role in our argument.
\begin{lem}\label{lem1}
	Let $S, T\in M_n$. Then for any $t\in (0,\infty)$,
	\[
	2\left|\tr \,   S^*T\right|\le  t\tr \,   S^*S+\frac{1}{t}\tr \,   T^*T. 
	\]
\end{lem}
\begin{proof}
	By the matrix-form Cauchy-Schwarz inequality \cite[p. 95]{Bha97}, we have
	\begin{eqnarray}\label{e2.1}
		2\left|\tr \,   S^*T\right|\le 2\left( \tr \,   S^*S\right)^\frac{1}{2}\left( \tr \,   T^*T\right)^\frac{1}{2}.
	\end{eqnarray}
	 By the scalar Arithmetic Mean-Geometric Mean (AM-GM) inequality, we have for any  $t\in (0,\infty)$,
	 \begin{eqnarray}\label{e2.2}
	 	2\left( \tr \,   S^*S\right)^\frac{1}{2}\left( \tr \,   T^*T\right)^\frac{1}{2}\le t\tr \,   S^*S+\frac{1}{t}\tr \,   T^*T.
	 \end{eqnarray}
	 Combing (\ref{e2.1}) and (\ref{e2.2}) gives Lemma \ref{lem1}.
\end{proof}
\begin{lem}\label{lem2}
	Let $X,Y$ be two $n\times n$ positive semidefinite matrices. Assume $Q$ is a contraction of order $n$, i.e., $\|Q\|_\infty\le 1$. Then for any $t\in (0,\infty)$,
	\[
	4\left|\tr \,   QXY\right|\le t\tr \,  (X^2+Y^2)+\dfrac{1}{t}\tr \,  (XY+YX).
	\]
\end{lem}
\begin{proof}
	\begin{eqnarray*}
		4\left|\tr \,   QXY\right|&=& 4\left|\tr \,   (Y^\frac{1}{2}QX^\frac{1}{2})(X^\frac{1}{2}Y^\frac{1}{2})\right|\\
		&\le& 2t\tr \,   YQXQ^*+\dfrac{2}{t}\tr \,   XY \ (\text{By Lemma \ref{lem1}})\\
		&\le& t\left( \tr \,   Q^*Y^2Q+\tr \,   QX^2Q^*\right) +\dfrac{1}{t}\tr \,  (XY+YX)\\
		&&   (\text{Take the parameter in Lemma \ref{lem1} as $1$})\\
		&\le& t\tr \,  (X^2+Y^2)+\dfrac{1}{t}\tr \,  (XY+YX) \ (\text{Since $Q$ is a contraction}).
	\end{eqnarray*}
	\end{proof}
\noindent\emph{Proof of Conjecture \ref{thm}.}	Let $A=U\left|A\right|$ and $B=V\left|B\right|$ be Polar decompositions of $A, B$, respectively. Denote $W=U^*V$. Then we have for any $t\in (0, \infty)$,
	\begin{eqnarray*}
		2\|A+B\|_F^2&=&  2 \tr \,  \left(\left|A\right|^2+\left|B\right|^2\right) +4\Re \left( \tr \,    W\left|B\right|\left|A\right|\right) \\
		&\le& 2 \tr \,  \left(\left|A\right|^2+\left|B\right|^2\right) + 4\left|\tr \,    W\left|B\right|\left|A\right|\right| \\
	    &\le& (2+t) \tr \,  \left(\left|A\right|^2+\left|B\right|^2\right)+\dfrac{1}{t}\tr \,  (\left|A\right|\left|B\right|+\left|B\right|\left|A\right|)\\ &&(\text{Set $Q=W, X=\left|B\right|, Y=\left|A\right|$ in Lemma \ref{lem2}}).
	\end{eqnarray*}
	Taking $t=\sqrt{2}-1$ in the above inequality yields 
	\[
	2\|A+B\|_F^2\le (\sqrt{2}+1)\left\|\left|A\right|+\left|B\right|\right\|_F^2,
	\]
	which is equivalent to (\ref{e1}).
\qed
	\section*{Acknowledgments}  
The author is grateful to Professor M. Lin, his academic supervisor, for bringing this problem to his attention.

\end{document}